\documentclass[12pt]{amsart}
\usepackage{graphicx}
\usepackage[headings]{fullpage}
\usepackage{amssymb,epic,eepic,epsfig,amsbsy,amsmath,amscd,color}
\numberwithin{equation}{section}
                        \theoremstyle{plain}
\usepackage{mathrsfs}

\newcommand\no[1]{}

\newtheorem{theorem}{Theorem}[section]
\newtheorem{thm}{Theorem}
\newtheorem{lemma}[theorem]{Lemma}

\newtheorem{proposition}[theorem]{Proposition}

\newtheorem*{conjecture*}{Conjecture}
\newtheorem*{theorem*}{Theorem}

\theoremstyle{definition}

\def\BZ{\mathbb Z}

\def\be { \begin{equation} }
\def\ee { \end{equation} }

\begin{document}
\allowdisplaybreaks
\baselineskip16pt
\title{Genera and crossing numbers of $2$-bridge knots}

\begin{abstract}
In this paper, we determine the average genus of all the $2$-bridge knots with a given crossing number. 
As a consequence, we obtain the oblique asymptote of this value as the crossing number grows. 
\end{abstract}

\author{Masaaki Suzuki}
\address{Department of Frontier Media Science,
  Meiji University,
  4-21-1 Nakano, Nakano-ku, Tokyo, 164-8525, Japan}
\email{mackysuzuki@meiji.ac.jp}

\author{Anh T. Tran}
\address{Department of Mathematical Sciences,
The University of Texas at Dallas,
Richardson, TX 75080-3021, USA}
\email{att140830@utdallas.edu}



\maketitle

\section{Introduction}

Dunfield provided experimental data in \cite{dunfield} 
which suggests that the genus of any knot grows linearly with respect to its crossing number. 
It is known that the family of $2$-bridge knots is a fundamental class of knots. 
Many knot properties are first studied for $2$-bridge knots before for the general case. 

Let  $\overline{g}_{c}$ be the average genus of all the $2$-bridge knots with $c$ crossings. 
Motivated by Dunfield's result, Cohen in \cite{cohen} \cite{cohen2} gave lower and upper bounds on $\overline{g}_{c}$ in terms of $c$. 
In this paper, we will determine the explicit value of $\overline{g}_{c}$. 
As a consequence, we can obtain the oblique asymptote of $\overline{g}_{c}$. 
That is to say, the following is the main theorem of this paper. 

\begin{thm}\label{mainthm}
For $c \geq 3$, we have
\[
\overline{g}_{c} =
\left\{
\begin{array}{ll}
\displaystyle{\frac{1}{4} c + \frac{1}{12} +  \frac{c-5}{2^c - 4}} & c \equiv 0 \pmod{2} 
\smallskip \\
\displaystyle{\frac{1}{4} c + \frac{1}{12} +  \frac{1}{3 \cdot 2^{(c-3)/2}}} & c \equiv 1 \pmod{4} 
\smallskip\\
\displaystyle{\frac{1}{4} c + \frac{1}{12} +  \frac{2^{(c+1)/2} - 3 c + 11}{12 \left(2^{c-3} + 2^{(c-3)/2} + 1\right)}} & c \equiv 3 \pmod{4}
\end{array}
\right. .
\]
In particular, $\overline{g}_c \sim \displaystyle{\frac{1}{4} c + \frac{1}{12}}$ as $c \to \infty$. 
\end{thm}
This result supports Dunfield's observation affirmatively. 
Cohen and Lowrance \cite{cohen3} obtained the same asymptote independently. 

\section{Preliminary}

In this section, we review some known facts about $2$-bridge knots from \cite{murasugi}, \cite{cromwell}, 
and show some fundamental formulas about binomial coefficients. 

It is known that a $2$-bridge knot corresponds to a rational number.  
We denote by $K([a_1,a_2, \ldots, a_n])$ the $2$-bridge knot determined by a continued fraction 
\[
[a_1, a_2, \ldots, a_n] = 
\frac{1}{a_1 + \frac{1}{a_2 + \frac{1}{\ddots \frac{1}{a_n}}}} . 
\]
The genus of the $2$-bridge knot corresponding to a rational number is a half of the length of the even continued fraction. 
Namely, if we take the even continued fraction $[2 a_1, 2 a_2, \ldots, 2 a_{2m}]$ of a rational number where $a_i \neq 0$, 
then the genus of $K = K([2 a_1,2 a_2, \ldots, 2 a_{2m}])$ is $m$. 
Remark that the lengths of the even continued fractions for $2$-bridge knots are always even. 
Moreover, the crossing number $c$ of $K$ is given by 
\[
c = \left(\sum_{i=1}^{2m} 2|a_i| \right) - \ell 
\]
where $\ell$ is the number of sign changes in the sequence $(2 a_1, 2 a_2, \ldots, 2 a_{2m})$ (See \cite{SZK1}, \cite{suzukitran}, and \cite{SZK2}). 
Since $|a_i| \ge 1$ and $0 \le \ell \le 2m-1$, 
we have $4m \le \sum_{i=1}^{2m} 2|a_i| =c + \ell \le c + (2m-1)$. 
This implies that $2m \le c-1$ and then $0 \le \ell \le 2m-1 \le c-2$. 
Since $\sum_{i=1}^{2m} 2|a_i|$ is even, $\ell$ and $c$ have the same parity. 
For each $\mathbf{b}=(b_1, b_2, \dots, b_{2m})$ where $b_i \geq 1$,  
let $S_{\mathbf{b},\ell} =S_{(b_1, b_2, \dots, b_{2m}),\ell}$ be the set of all the $2$-bridge knots $K([2a_1, 2a_2, \dots, 2a_{2m}])$ 
such that $a_i = \pm b_i$ and that the number of sign changes in the sequence $(a_1, a_2, \dots, a_{2m})$ is exactly $\ell$. 
For example, we have 
\[
S_{(1,1),0} = \{ K([2,2] \}, \quad S_{(1,1),1} = \{ K([2,-2], K([-2,2]) \} . 
\]

Note that $K([2a_1, 2a_2, \dots, 2a_{2m}])$ and $K([-2a_{2m}, -2a_{2m-1}, \dots, -2a_{1}])$ represent the same $2$-bridge knot. 
Moreover, $K([-2a_1, -2a_2, \dots, -2a_{2m}])=K([2a_{2m}, 2a_{2m-1}, \dots, 2a_{1}])$ 
is the mirror image of $K([2a_1, 2a_2, \dots, 2a_{2m}]) = K([-2a_{2m}, -2a_{2m-1}, \dots, -2a_{1}])$. 
These two knots are distinct if and only if  $(2a_1, 2a_2, \dots, 2a_{2m})$ is not symmetric. 

In \cite{ernstsumners}, Ernst and Sumners determined the number of $2$-bridge knots with respect to crossing number. 

\begin{theorem}[Ernst-Sumners \cite{ernstsumners}]\label{thmernstsumners}
For $c \geq 3$, the number of all the $2$-bridge knots with $c$ crossings is given by 
\[
TK(c) = 
\left\{
\begin{array}{ll}
\displaystyle{\frac{1}{3} (2^{c-2} - 1)} & c \equiv 0 \pmod{2} 
\smallskip\\
\displaystyle{\frac{1}{3} (2^{c-2} + 2^{(c-1)/2})} & c \equiv 1 \pmod{4} 
\smallskip\\ 
\displaystyle{\frac{1}{3} (2^{c-2} + 2^{(c-1)/2} + 2)} & c \equiv 3 \pmod{4} 
\end{array}
\right. .
\]
\end{theorem}

In this paper, we make use of several identities on binomial coefficients in order to prove Theorem \ref{mainthm}. 
In addition to well-known formulas:
\[
\binom{a}{b} \binom{b}{c} = \binom{a}{c} \binom{a-c}{b-c}, \quad 
\sum_{q=0}^{n}  \binom{n}{q} = 2^n, \quad 
\sum_{q=0}^{\lfloor n/2 \rfloor}  \binom{n}{2q} = 2^{n-1} , \quad 
\sum_{q=0}^{n}  q \binom{n}{q} = n 2^{n-1}, 
\]
we prepare the following.  

\begin{lemma}\label{lem22} 
We have
$$
\sum_{q=0}^{n-1}  x^{q} \binom{2n-1-q}{q} = \frac{\left( \frac{2x+1+ \sqrt{4x+1}}{2} \right)^n - \left( \frac{2x+1- \sqrt{4x+1}}{2} \right)^n}{\sqrt{4x+1}},
$$
and
$$
\sum_{q=0}^{n}  x^{q} \binom{2n-q}{q} = \frac{ \frac{\sqrt{4x+1} +1}{2} \left( \frac{2x+1+ \sqrt{4x+1}}{2} \right)^n + \frac{\sqrt{4x+1} -1}{2}  \left( \frac{2x+1- \sqrt{4x+1}}{2} \right)^n}{\sqrt{4x+1}}.
$$
In particular, it follows that 
$$
\sum_{q=0}^{n-1}  2^{q} \binom{2n-1-q}{q} = \frac{4^n-1}{3} \qquad \text{and} \qquad \sum_{q=0}^{n}  2^{q} \binom{2n-q}{q} = \frac{2\cdot 4^n + 1}{3}.
$$
\end{lemma}

\begin{proof}
Let $\alpha_n = \sum_{q=0}^{n-1}  x^{q} \binom{2n-1-q}{q}$ and $\beta_n = \sum_{q=0}^{n}  x^{q} \binom{2n-q}{q}$. Then
\begin{eqnarray*}
\alpha_{n+1} - \beta_n &=& \sum_{q=0}^{n}  x^{q} \left[ \binom{2n+1-q}{q} - \binom{2n-q}{q}\right]\\
&=&  \sum_{q=1}^{n}  x^{q} \binom{2n-q}{q-1} \\
&=& x\sum_{q=1}^{n}  x^{q-1} \binom{2n-1-(q-1)}{q-1} \\
 &=& x \sum_{q'=0}^{n-1}  x^{q'} \binom{2n-1-q'}{q'} \\
  &=& x\alpha_n.
\end{eqnarray*}
This implies that $\beta_n =  \alpha_{n+1} - x \alpha_n$.

Similarly 
\begin{eqnarray*}
\beta_n - \alpha_{n} &=& \sum_{q=0}^{n}  x^{q} \left[ \binom{2n-q}{q} - \binom{2n-1-q}{q} \right]\\
&=&  \sum_{q=1}^{n}  x^{q} \binom{2n-1-q}{q-1} \\
&=& x \sum_{q=1}^{n}  x^{q-1} \binom{2n-2-(q-1)}{q-1} \\
 &=& x \sum_{q'=0}^{n-1}  x^{q'} \binom{2n-2-q'}{q'} \\
  &=& x \beta_{n-1}.
\end{eqnarray*}
This implies that
$$
\alpha_n = \beta_n - x \beta_{n-1} = (\alpha_{n+1} - x \alpha_n) - x(\alpha_{n} - x \alpha_{n-1}) = \alpha_{n+1}  - 2x \alpha_{n} + x^2 \alpha_{n-1}.
$$
Hence $\alpha_{n+1}  - (2x+1) \alpha_{n} + x^2 \alpha_{n-1}=0$. Since $\alpha_0=0$ and $\alpha_1=1$, we obtain
$$
\alpha_n = \frac{\left( \frac{2x+1+ \sqrt{4x+1}}{2} \right)^n - \left( \frac{2x+1- \sqrt{4x+1}}{2} \right)^n}{\sqrt{4x+1}}.
$$
Then 
\begin{eqnarray*}
\beta_n &=&  \alpha_{n+1} - x \alpha_n \\
&=& \frac{\left( \frac{2x+1+ \sqrt{4x+1}}{2} \right)^{n+1} - \left( \frac{2x+1- \sqrt{4x+1}}{2} \right)^{n+1}}{\sqrt{4x+1}} - x \frac{\left( \frac{2x+1+ \sqrt{4x+1}}{2} \right)^n - \left( \frac{2x+1- \sqrt{4x+1}}{2} \right)^n}{\sqrt{4x+1}} \\
&=& \frac{ \frac{\sqrt{4x+1} +1}{2} \left( \frac{2x+1+ \sqrt{4x+1}}{2} \right)^n + \frac{\sqrt{4x+1} -1}{2}  \left( \frac{2x+1- \sqrt{4x+1}}{2} \right)^n}{\sqrt{4x+1}}.
\end{eqnarray*}
\end{proof}

By using the above lemma, we obtain the following identities, which play an important role in the proof of Theorem \ref{mainthm}. 

\begin{lemma}\label{lem23}
We have 
\begin{eqnarray*}
\sum_{q=0}^{n-1}  q \, 2^{q} \binom{2n-1-q}{q}  &=& \frac{2}{27} \left( (4^n-1) (3 n- 2)- 3n \right), \\
\sum_{q=0}^{n}  q \, 2^{q} \binom{2n-q}{q}  &=& \frac{2}{27} \left( (4^n-1)  (6n -1)+ 12n \right).
\end{eqnarray*}
\end{lemma}

\begin{proof}
Let $a= \frac{2x+1+ \sqrt{4x+1}}{2}$ and $b= \frac{2x+1- \sqrt{4x+1}}{2}$.  
Then their derivatives with respect to $x$ are $a' = 1 + \frac{1}{\sqrt{4x+1}}$ and $b' = 1 - \frac{1}{\sqrt{4x+1}}$. 
Since $\alpha_n =  \frac{a^n-b^n}{a-b}$, the derivative of $\alpha_n$ with respect to $x$ is given by  
\begin{eqnarray*}
\alpha'_n &=& \frac{n(a^{n-1} a' - b^{n-1} b')(a-b) - (a^n - b^n)(a' - b')}{(a-b)^2} \\
&=& \frac{n(a^{n-1}  - b^{n-1} + (a^{n-1} + b^{n-1}) \frac{1}{\sqrt{4x+1}})\sqrt{4x+1}- (a^n - b^n) \frac{2}{\sqrt{4x+1}}}{4x+1} \\
&=& \frac{n((a^{n-1}  - b^{n-1}) \sqrt{4x+1} + a^{n-1} + b^{n-1})- (a^n - b^n) \frac{2}{\sqrt{4x+1}}}{4x+1}.
\end{eqnarray*}
Hence 
\begin{eqnarray*}
\alpha'_n |_{x=2} &=& \frac{n[3(4^{n-1}  - 1)  + 4^{n-1} + 1]- (4^n - 1) \frac{2}{3}}{9} \\
&=& \frac{1}{27}((4^n-1) (3n - 2) - 3 n).
\end{eqnarray*}
By $\sum_{q=0}^{n-1}  q \, 2^{q} \binom{2n-1-q}{q} = 2 \alpha'_n |_{x=2}$, we obtain the first desired equality. 

Since $\beta_n =  \alpha_{n+1} - x \alpha_n$, we have
$$
\beta'_n =  \alpha'_{n+1} - x \alpha'_n - \alpha_n.
$$
Hence 
\begin{eqnarray*}
\beta'_n |_{x=2} &=&  \frac{1}{9} \left( (4^{n+1}-1) (n+ \frac{1}{3})-n-1 \right) -  \frac{2}{9}  \left( (4^n-1) (n- \frac{2}{3})-n \right) - \frac{4^n-1}{3} \\
&=& \frac{1}{27}((4^n-1) (6n - 1)+ 12).
\end{eqnarray*}
By $\sum_{q=0}^{n}  q \, 2^{q} \binom{2n-q}{q} = 2 \beta'_n |_{x=2}$, we also obtain the second desired equality. 
\end{proof}

\section{Total genus}

In this section, we compute the total genus of $2$-bridge knots with a given crossing number, which is denoted by $TG(c)$ for a given crossing number $c$. By definition, we have
\[
TG(c) = \sum_g g \cdot \sharp \{K \mbox{: $2$-bridge knot} \, | \, c(K) = c, g(K) = g \},
\] 
where $c(K)$ and $g(K)$ are the crossing number and the genus of a knot $K$ respectively.

\begin{proposition}\label{totalgenus}
For $c \geq 3$, we have 
\[
TG(c) =
\left\{
\begin{array}{ll}
\displaystyle{\frac{1}{36} \left( (3 c + 1) \cdot 2^{c-2} - 16 \right)} & c \equiv 0 \pmod{2} 
\smallskip\\
\displaystyle{\frac{1}{36} \left( (3 c + 1) \cdot 2^{c-2} + (3 c + 5)\cdot 2^{(c-1)/2} + 8 \right)} & c \equiv 1 \pmod{4} 
\smallskip\\
\displaystyle{\frac{1}{36} \left( (3 c + 1) \cdot 2^{c-2} + (3 c + 5)\cdot 2^{(c-1)/2} + 24 \right)} & c \equiv 3 \pmod{4} 
\end{array}
\right. .
\]
\end{proposition} 

Combining Theorem \ref{thmernstsumners} and Proposition \ref{totalgenus}, we obtain Theorem \ref{mainthm}. 
We will divide the proof of Proposition \ref{totalgenus} into the cases of even crossing number and odd crossing number. 

\subsection{Even crossing number}

We first consider the case $c=2k$, $k \in \BZ$. 
Since $c$ and $\ell$ have the same parity, we may write $\ell = 2l$, where $l \in \BZ$ and $0 \le l \le k-1$. 
For every $0 \le l \le k-1$, the equation 
$$\sum_{i=1}^{2m} b_i = k+l,$$
where $b_i \in \BZ_{\ge 1}$, has exactly $\binom{k+l-1}{2m-1}$ solutions $(b_1, b_2, \dots, b_{2m})$. 
The number of ways of choosing $a_i = \pm b_i$ such that the number of sign changes among $(a_1, a_2, \dots, a_{2m})$ is exactly $\ell=2l$ 
is equal to twice the number of solutions $(d_1, d_2, \dots, d_{\ell+1}) \in (\BZ_{\ge 1})^{\ell+1}$ of the equation
$$\sum_{i=1}^{\ell+1} d_i = 2m,$$
and so it is equal to $2\binom{2m-1}{\ell} = 2\binom{2m-1}{2l}$. 

Suppose that $\mathbf{b}=(b_1, b_2, \dots, b_{2m})$ is not symmetric, namely $(b_1, b_2, \dots, b_{2m}) \neq (b_{2m}, \dots, b_2, b_1)$. 
If $K([2a_1, 2a_2, \dots, 2a_{2m}])$ belongs to $S_{\mathbf{b},\ell}$, 
then 
$K([2a_{2m}, 2a_{2m-1}, \dots, 2a_{1}])$ does not belong to $S_{\mathbf{b},\ell}$. 
Since the two knots $K([2a_1, 2a_2, \dots, 2a_{2m}])$ and $K([-2a_1, -2a_2, \dots, -2a_{2m}])$ in $S_{\mathbf{b},\ell}$ are distinct knots. 
Hence $S_{\mathbf{b},\ell}$ has cardinality $2\binom{2m-1}{\ell} = 2\binom{2m-1}{2l}$.  
Remark that if $(b_{2m}, b_{2m-1}, \dots, b_{1}) = (b'_1, b'_2, \dots, b'_{2m})$, then $S_{(b_1, b_2, \dots, b_{2m}),\ell} = S_{(b'_1, b'_2, \dots, b'_{2m}),\ell}$.


Suppose that $\mathbf{b}=(b_1, b_2, \dots, b_{2m})$ is symmetric 
(this only occurs when $k+l$ is even and the number of symmetric cases is equal to $\binom{(k+l)/2-1}{m-1}$). 
If $K([2a_1, 2a_2, \dots, 2a_{2m}])$ belongs to $S_{\mathbf{b},\ell}$, 
then both $K([2a_{2m}, 2a_{2m-1}, \dots, 2a_{1}])$ and $K([-2a_{2m}, -2a_{2m-1}, \dots, -2a_{1}])$ also belong to $S_{\mathbf{b},\ell}$. 
We need to consider the two cases where $(2a_1, 2a_2, \dots, 2a_{2m})$ is symmetric and not.  
If $(2a_1, 2a_2, \dots, 2a_{2m})$ is not symmetric, 
then $K([2a_1, 2a_2, \dots, 2a_{2m}]) = K([-2a_{2m}, -2a_{2m-1}, \dots, -2a_{1}])$ and $K([-2a_1, -2a_2, \dots, -2a_{2m}])=K([2a_{2m}, 2a_{2m-1}, \dots, 2a_{1}])$ are distinct knots. 
However, 
if $(2a_1, 2a_2, \dots, 2a_{2m})$ is symmetric, namely  $(2a_1, 2a_2, \dots, 2a_{2m})=(2a_{2m}, 2a_{2m-1}, \dots, 2a_1)$ 
(the number of symmetric cases is equal to $2\binom{m-1}{l}$), 
then $K([2a_1, 2a_2, \dots, 2a_{2m}]) = K([-2a_{2m}, -2a_{2m-1}, \dots, -2a_{1}])$ and $K([-2a_1, -2a_2, \dots, -2a_{2m}])=K([2a_{2m}, 2a_{2m-1}, \dots, 2a_{1}])$ are the same knot. 
Note that $(2a_1, 2a_2, \dots, 2a_{2m}) \not= (-2a_{2m}, -2a_{2m-1}, \dots, -2a_1)$. 
(Otherwise, we have $(2a_1, 2a_2, \dots, 2a_{2m}) = (2a_1, 2a_2, \dots, 2a_{m}, -2a_m, \dots, -2a_2, -2a_1)$ which implies that the number of sign changes is an odd number.)
Therefore if $\mathbf{b}$ is symmetric, $S_{\mathbf{b},\ell}$ has cardinality $\binom{2m-1}{2l}$. 
Remark that in this case, $S_{\mathbf{b},\ell} = S_{\mathbf{b}',\ell}$ if and only if $\mathbf{b} = \mathbf{b}'$. 

If 
$k+l$ is odd, then $(b_1, b_2, \dots, b_{2m})$ is not symmetric.   
Therefore the number of $2$-bridge knots with crossing number $c=2k$ and  number of sign changes $\ell = 2l$ is equal to
$$A_l =  \sum_{m=l+1}^{(k+l)/2} \binom{k+l-1}{2m-1} \cdot 2 \binom{2m-1}{2l} \cdot \frac{1}{2} = \sum_{m=l+1}^{(k+l)/2} \binom{k+l-1}{2l} \binom{k-l-1}{2m-2l-1}$$
and the total genus, namely, the total number of the genera of these $2$-bridge knots is 
$$B_l 
= \sum_{m=l+1}^{(k+l)/2} m \binom{k+l-1}{2l} \binom{k-l-1}{2m-2l-1}.$$

If 
$k+l$ is even, then there are both cases where $(b_1, b_2, \dots, b_{2m})$ is symmetric and not symmetric. 
Therefore the number of $2$-bridge knots with crossing number $c=2k$ and number of sign changes $\ell = 2l$ is equal to
\begin{eqnarray*}
A_l &=&  \sum_{m=l+1}^{(k+l)/2}  \left[ \binom{k+l-1}{2m-1}  - \binom{(k+l)/2-1}{m-1} \right] \binom{2m-1}{2l}+ \binom{(k+l)/2-1}{m-1} \binom{2m-1}{2l} \\
&=& \sum_{m=l+1}^{(k+l)/2} \binom{k+l-1}{2m-1} \binom{2m-1}{2l} = \sum_{m=l+1}^{(k+l)/2} \binom{k+l-1}{2l} \binom{k-l-1}{2m-2l-1}
\end{eqnarray*}
and the total genus is 
$$B_l 
= \sum_{m=l+1}^{(k+l)/2} m \binom{k+l-1}{2l} \binom{k-l-1}{2m-2l-1}.$$
As a consequence, $A_l$ and $B_l$ are the same regardless of the parity of $k+l$.  

In order to simplify $B_l$, we compute  $2B_l-(2l+1)A_l$ and $A_l$ as follows: 
\begin{eqnarray*}
2B_l-(2l+1)A_l &=&  \sum_{m=l+1}^{(k+l)/2} (2m-2l-1) \binom{k+l-1}{2l} \binom{k-l-1}{2m-2l-1} \\
&=&   \binom{k+l-1}{2l} \sum_{m=l+1}^{(k+l)/2} (2m-2l-1)  \binom{k-l-1}{2m-2l-1} \\
&=&  \binom{k+l-1}{k-l-1} \sum_{m'=0}^{(k-l)/2-1} (2m'+1) \binom{k-l-1}{2m'+1}.
\end{eqnarray*}
If $0 \le l \le k-2$, then 
\begin{eqnarray*}
A_l &=&   \binom{k+l-1}{2l} \sum_{m=l+1}^{(k+l)/2} \binom{k-l-1}{2m-2l-1} \\
&=&  \binom{k+l-1}{k-l-1} \sum_{m'=0}^{(k-l)/2-1} \binom{k-l-1}{2m'+1} \\
&=&  \frac{1}{2} \binom{k+l-1}{k-l-1} \sum_{m''=0}^{k-l-1} \binom{k-l-1}{m''} \\
&=& 2^{k-l-2} \binom{k+l-1}{k-l-1}.
\end{eqnarray*}
If $l=k-1$, then $A_l=0= 2^{k-l-2} \binom{k+l-1}{k-l-1} 2^{k-l-2} - \frac{1}{2}$.
When we take the sum of $B_l$ with respect to $l$, 
$B_l$ depends on the three cases, $0 \le l \le k-3$, $l=k-2$, and $l=k-1$. 
If $0 \le l \le k-3$, then 
\begin{eqnarray*}
2B_l &=& 2B_l-(2l+1)A_l + (2l+1)A_l \\
&=& \frac{1}{2}  \binom{k+l-1}{k-l-1} \sum_{m''=0}^{k-l-1} m'' \binom{k-l-1}{m''} + (2l+1) 2^{k-l-2} \binom{k+l-1}{k-l-1} \\
&=& \frac{k-l-1}{2} 2^{k-l-2}  \binom{k+l-1}{k-l-1} + (2l+1) 2^{k-l-2} \binom{k+l-1}{k-l-1}  \\
&=& (k+3l+1)2^{k-l-3} \binom{k+l-1}{k-l-1} .
\end{eqnarray*}
If $l=k-2$, then
\begin{eqnarray*}
2B_l &=& 2B_l-(2l+1)A_l + (2l+1)A_l \\
&=& \binom{k+l-1}{k-l-1} + (2l+1) 2^{k-l-2} \binom{k+l-1}{k-l-1} \\
&=& (k-l-1)2^{k-l-3} \binom{k+l-1}{k-l-1} +  \frac{2k-3}{2} + (2l+1) 2^{k-l-2} \binom{k+l-1}{k-l-1} \\
&=& (k+3l+1)2^{k-l-3} \binom{k+l-1}{k-l-1} +  \frac{2k-3}{2}.
\end{eqnarray*}
If $l=k-1$, then 
\begin{eqnarray*}
2B_l &=& 2B_l-(2l+1)A_l + (2l+1)A_l \\
&=& 0 + (2l+1) \left[ 2^{k-l-2} \binom{k+l-1}{k-l-1}  - \frac{1}{2} \right] \\
&=& (k-l-1)2^{k-l-3} \binom{k+l-1}{k-l-1}+ (2l+1) 2^{k-l-2} \binom{k+l-1}{k-l-1} - \frac{2k-1}{2} \\
&=& (k+3l+1)2^{k-l-3} \binom{k+l-1}{k-l-1} -  \frac{2k-1}{2}.
\end{eqnarray*}

By Lemma \ref{lem22}, Lemma \ref{lem23} and $c=2k$, we conclude 
\begin{eqnarray*}
TG(c) &=&  \sum_{l=0}^{k-1} B_l \\
&=& \frac{1}{2} \sum_{l=0}^{k-1}  (k+3l+1) 2^{k-l-3} \binom{k+l-1}{k-l-1} + \frac{1}{2} \left( \frac{2k-3}{2} - \frac{2k-1}{2} \right) \\
&=&  \frac{1}{8} \sum_{l=0}^{k-1}  (k+3l+1) 2^{k-l-1} \binom{k+l-1}{k-l-1} - \frac{1}{2} \\
 &=&  \frac{2k-1}{4} \sum_{l=0}^{k-1}  2^{k-l-1} \binom{k+l-1}{k-l-1} - \frac{3}{8}\sum_{l=0}^{k-1}  (k-l-1) 2^{k-l-1} \binom{k+l-1}{k-l-1} - \frac{1}{2} \\
 &=& \frac{2k-1}{4} \sum_{q=0}^{k-1}  2^{q} \binom{2k-2-q}{q} - \frac{3}{8} \sum_{q=0}^{k-1}  q \, 2^{q} \binom{2k-2-q}{q} - \frac{1}{2} \\
 &=& \frac{2k-1}{4} \, \frac{2 (4^{k-1}-1)+1}{3} - \frac{3}{8} \, \frac{2}{27} \left( (4^{k-1}-1)(6k-7) + 12k-12 \right)- \frac{1}{2} \\
&=& \frac{c-1}{4} \, \frac{2 (2^{c-2}-1)+1}{3} - \frac{1}{36} \left( (2^{c-2}-1)(3c-7) + 6c-12 \right)- \frac{1}{2} \\
&=& \frac{1}{36} \left( (3c+1) \cdot 2^{c-2} - 16 \right) .
\end{eqnarray*}

Therefore Proposition \ref{totalgenus} for the case of even crossing number is proved. 

\subsection{Odd crossing number} 
We now consider the case $c=2k+1$, $k \in \BZ$. 
Similarly, we may write $\ell = 2l+1$, where $l \in \BZ$ and $0 \le l \le k-1$. For every $0 \le l \le k-1$, the equation 
$$\sum_{i=1}^{2m} b_i = k+l+1,$$
where $b_i \in \BZ_{\ge 1}$, has exactly $\binom{k+l}{2m-1}$ solutions $(b_1, b_2, \dots, b_{2m})$. 
The number of ways of choosing $a_i = \pm b_i$ such that the number of sign changes among $a_i$'s is exactly $\ell=2l+1$ is equal to twice the number of solutions $(d_1, d_2, \dots, d_{\ell+1}) \in (\BZ_{\ge 1})^{\ell+1}$ of the equation
$$\sum_{i=1}^{\ell+1} d_i = 2m,$$
and so it is equal to $2\binom{2m-1}{\ell} = 2\binom{2m-1}{2l+1}$. 

Since the number of sign changes is odd,  $(2a_1, 2a_2, \dots, 2a_{2m})$ is not symmetric. 
Then $K([2a_1, 2a_2, \dots, 2a_{2m}]) = K([-2a_{2m}, -2a_{2m-1}, \dots, -2a_{1}])$ 
and $K([2a_{2m}, 2a_{2m-1}, \dots, 2a_{1}])=K([-2a_1, -2a_2, \dots, -2a_{2m}])$  are always distinct knots. 


Suppose that $(b_1, b_2, \dots, b_{2m})$ is not symmetric, then
$S_{(b_1, b_2, \dots, b_{2m}),\ell}$ has cardinality $2\binom{2m-1}{\ell} = 2\binom{2m-1}{2l+1}$.
Remark that if $(b_{2m}, b_{2m-1}, \dots, b_{1}) = (b'_1, b'_2, \dots, b'_{2m})$, then $S_{(b_1, b_2, \dots, b_{2m}),\ell} = S_{(b'_1, b'_2, \dots, b'_{2m}),\ell}$.

Suppose that $(b_1, b_2, \dots, b_{2m})$ is symmetric (this only occurs when $k+l+1$ is even and the number of symmetric cases is equal to $\binom{(k+l-1)/2}{m-1}$).  
Note that $(-2a_{2m}, -2a_{2m-1},\dots, -2a_1) = (2a_1, 2a_2, \dots, 2a_{2m})$ if and only if $(2a_1, 2a_2, \dots, 2a_{2m}) = (2a_1, 2a_2, \dots, 2a_{m}, -2a_m, \dots, -2a_2, -2a_1)$.  
Therefore $S_{(b_1, b_2, \dots, b_{2m}),\ell}$ has cardinality 
$\frac{1}{2} (2 \binom{2m-1}{\ell} - 2 \binom{m-1}{l}) + 2 \binom{m-1}{l} = \binom{2m-1}{\ell} + \binom{m-1}{l} = \binom{2m-1}{2l+1} + \binom{m-1}{l}$. 

If 
$k+l$ is even, then $(b_1, b_2, \dots, b_{2m})$ is not symmetric.   
Therefore the number of $2$-bridge knots with crossing number $c=2k+1$ and number of sign changes $\ell = 2l+1$ is equal to
$$A_l =  \sum_{m=l+1}^{(k+l)/2} \binom{k+l}{2m-1}\cdot 2  \binom{2m-1}{2l+1} \cdot \frac{1}{2} = \sum_{m=l+1}^{(k+l)/2} \binom{k+l}{2l+1} \binom{k-l-1}{2m-2l-2}$$
and the total genus is 
$$B_l =  \sum_{m=l+1}^{(k+l)/2} m \binom{k+l}{2l+1} \binom{k-l-1}{2m-2l-2}.$$

If 
$k+l$ is odd, then there are both cases where $(b_1, b_2, \dots, b_{2m})$ is symmetric and not symmetric. 
Therefore the number of $2$-bridge knots with crossing number $c=2k+1$ and the number of sign changes $\ell = 2l+1$ is equal to
\begin{eqnarray*}
A_l &=&  \sum_{m=l+1}^{(k+l+1)/2}  \left[ \left[ \binom{k+l}{2m-1} - \binom{(k+l-1)/2}{m-1}\right] \binom{2m-1}{2l+1} \right. \\
&& \qquad \qquad + \left. \binom{(k+l-1)/2}{m-1} \left[ \binom{2m-1}{2l+1}  + \binom{m-1}{l} \right] \right] \\
&=&  \sum_{m=l+1}^{(k+l+1)/2}  \left[ \binom{k+l}{2m-1}  \binom{2m-1}{2l+1} + \binom{(k+l-1)/2}{m-1} \binom{m-1}{l} \right] \\
&=&  \sum_{m=l+1}^{(k+l+1)/2}  \left[ \binom{k+l}{2l+1} \binom{k-l-1}{2m-2l-2} + \binom{(k+l-1)/2}{l} \binom{(k-l-1)/2}{m-l-1} \right]
\end{eqnarray*}
and the total genus is 
$$B_l =  \sum_{m=l+1}^{(k+l+1)/2}  m \left[ \binom{k+l}{2l+1} \binom{k-l-1}{2m-2l-2} + \binom{(k+l-1)/2}{l} \binom{(k-l-1)/2}{m-l-1} \right].$$

Hence we obtain the total genus  
\begin{eqnarray*}
TG(2k+1) = \sum_{l=0}^{k-1} B_l &=& \sum_{l=0}^{k-1} \sum_{m=l+1}^{(k+l+1)/2}   m \binom{k+l}{2l+1} \binom{k-l-1}{2m-2l-2} \\
&& + \,  \sum_{\substack{l=0 \\ l + k +1 \equiv 0 \pmod{2}}}^{k-1} \sum_{m=l+1}^{(k+l+1)/2}   m \binom{(k+l-1)/2}{l}  \binom{(k-l-1)/2}{m-l-1}.
\end{eqnarray*}

Moreover, we simplify this expression by using binomial coefficient formulas. 
When we take the sum of $B_l$ with respect to $l$ of the first term,
$B_l$ depends on the three cases, $0 \le l \le k-3$, $l=k-2$, and $l=k-1$. 
If $0 \le l \le k-3$, then 
\begin{eqnarray*}
&& \sum_{m=l+1}^{(k+l+1)/2}   m \binom{k+l}{2l+1} \binom{k-l-1}{2m-2l-2} \\ 
&=& \sum_{m=l+1}^{(k+l+1)/2}  (m-l-1) \binom{k+l}{2l+1} \binom{k-l-1}{2m-2l-2} + (l+1) \sum_{m=l+1}^{(k+l+1)/2} \binom{k+l}{2l+1} \binom{k-l-1}{2m-2l-2} \\
&=& \binom{k+l}{2l+1} \sum_{m=l+1}^{(k+l+1)/2} (m-l-1) \binom{k-l-1}{2m-2l-2} + (l+1)  \binom{k+l}{2l+1} \sum_{m=l+1}^{(k+l+1)/2} \binom{k-l-1}{2m-2l-2}\\
&=& \binom{k+l}{2l+1} \sum_{m'=0}^{(k-l-1)/2} m' \binom{k-l-1}{2m'} + (l+1)  \binom{k+l}{2l+1} \sum_{m'=0}^{(k-l-1)/2} \binom{k-l-1}{2m'}\\
&=& \frac{1}{4} \binom{k+l}{2l+1} \sum_{m''=0}^{k-l-1} m'' \binom{k-l-1}{m''} + \frac{l+1}{2} \binom{k+l}{2l+1} \sum_{m''=0}^{k-l-1} \binom{k-l-1}{m''} \\
&=& \frac{k-l-1}{4} 2^{k-l-2} \binom{k+l}{2l+1} + \frac{l+1}{2} 2^{k-l-1}  \binom{k+l}{2l+1} \\
&=& (k+3l+3) 2^{k-l-4} \binom{k+l}{2l+1}.
\end{eqnarray*}
If $l=k-2$, then 
\begin{eqnarray*}
&& \sum_{m=l+1}^{(k+l+1)/2}   m \binom{k+l}{2l+1} \binom{k-l-1}{2m-2l-2} \\
&=& \sum_{m=l+1}^{(k+l+1)/2}  (m-l-1) \binom{k+l}{2l+1} \binom{k-l-1}{2m-2l-2} + (l+1) \sum_{m=l+1}^{(k+l+1)/2} \binom{k+l}{2l+1} \binom{k-l-1}{2m-2l-2} \\
&=& \binom{k+l}{2l+1} \sum_{m=l+1}^{(k+l+1)/2} (m-l-1) \binom{k-l-1}{2m-2l-2} + (l+1) \binom{k+l}{2l+1} \sum_{m=l+1}^{(k+l+1)/2}  \binom{k-l-1}{2m-2l-2}\\
&=& 0 + \frac{l+1}{2}  2^{k-l-1}\binom{k+l}{2l+1} \\
&=& \frac{k-l-1}{4}  2^{k-l-2} \binom{k+l}{2l+1} -\frac{k-1}{2} + \frac{l+1}{2}  2^{k-l-1} \binom{k+l}{2l+1} \\
&=& (k+3l+3) 2^{k-l-4} \binom{k+l}{2l+1} -\frac{k-1}{2}.
\end{eqnarray*}
If $l=k-1$, then 
\begin{eqnarray*}
&& \sum_{m=l+1}^{(k+l+1)/2}   m \binom{k+l}{2l+1} \binom{k-l-1}{2m-2l-2} \\
&=& (m-l-1) \sum_{m=l+1}^{(k+l+1)/2}  \binom{k+l}{2l+1} \binom{k-l-1}{2m-2l-2} + (l+1) \sum_{m=l+1}^{(k+l+1)/2}  \binom{k+l}{2l+1} \binom{k-l-1}{2m-2l-2} \\
&=& \binom{k+l}{2l+1} \sum_{m=l+1}^{(k+l+1)/2} (m-l-1) \binom{k-l-1}{2m-2l-2} + (l+1) \binom{k+l}{2l+1} \sum_{m=l+1}^{(k+l+1)/2}  \binom{k-l-1}{2m-2l-2}\\
&=& 0 + (l+1) \\
&=& (k+3l+3) 2^{k-l-4} \binom{k+l}{2l+1} + \frac{k}{2}.
\end{eqnarray*}
Since $-\frac{k-1}{2} + \frac{k}{2} = \frac{1}{2}$, we obtain

\begin{eqnarray*}
\sum_{l=0}^{k-1} \sum_{m=l+1}^{(k+l+1)/2}   m \binom{k+l}{2l+1} \binom{k-l-1}{2m-2l-2} = \sum_{l=0}^{k-1} (k+3l+3) 2^{k-l-4} \binom{k+l}{2l+1} + \frac{1}{2}.
\end{eqnarray*}

Similarly, if $l + k+1 \equiv 0 \pmod{2}$, then
\begin{eqnarray*}
&& \sum_{m=l+1}^{(k+l+1)/2}   m \binom{(k+l-1)/2}{l} \binom{(k-l-1)/2}{m-l-1} \\
&=& \sum_{m=l+1}^{(k+l+1)/2} (m-l-1) \binom{(k+l-1)/2}{l} \binom{(k-l-1)/2}{m-l-1} \\
&& + (l+1) \sum_{m=l+1}^{(k+l+1)/2} \binom{(k+l-1)/2}{l} \binom{(k-l-1)/2}{m-l-1} \\
&=&  \binom{(k+l-1)/2}{l}  \sum_{m=l+1}^{(k+l+1)/2} (m-l-1) \binom{(k-l-1)/2}{m-l-1} \\
&& + (l+1) \binom{(k+l-1)/2}{l} \sum_{m=l+1}^{(k+l+1)/2} \binom{(k-l-1)/2}{m-l-1} \\
&=&  \binom{(k+l-1)/2}{l}  \sum_{m'=0}^{(k-l-1)/2} m' \binom{(k-l-1)/2}{m'} \\
&&+ (l+1) \binom{(k+l-1)/2}{l}  \sum_{m'=0}^{(k-l-1)/2} \binom{(k-l-1)/2}{m'}  \\
&=& \frac{k-l-1}{2} 2^{(k-l-3)/2} \binom{(k+l-1)/2}{l} + (l+1)  2^{(k-l-1)/2} \binom{(k+l-1)/2}{l} \\
&=& (k+3l+3) 2^{(k-l-5)/2} \binom{(k+l-1)/2}{l}.
\end{eqnarray*}

Then we obtain 
\begin{eqnarray*}
TG(2k+1) &=& \sum_{l=0}^{k-1} (k+3l+3) 2^{k-l-4} \binom{k+l}{2l+1} + \frac{1}{2}  \\
&&+ \sum_{\substack{l=0 \\ l + k +1 \equiv 0 \pmod{2}}}^{k-1} (k+3l+3) 2^{(k-l-5)/2} \binom{(k+l-1)/2}{l} \\
&=& \frac{1}{8} \sum_{l=0}^{k-1} (k+3l+3) 2^{k-l-1} \binom{k+l}{2l+1} + \frac{1}{2}  \\
&&+ \sum_{\substack{l=0 \\ l + k +1 \equiv 0 \pmod{2}}}^{k-1} \frac{k+3l+3}{4} 2^{(k-l-1)/2} \binom{(k+l-1)/2}{l}.
\end{eqnarray*}

Furthermore, the first term of the above expression can be simplified by using the following:
\begin{eqnarray*}
&&  \sum_{l=0}^{k-1}  (k+3l+3) 2^{k-l-1} \binom{k+l}{2l+1} \\
 &=&  \sum_{l=0}^{k-1}  (k+3l+3) 2^{k-l-1} \binom{k+l}{k-l-1} \\
 &=&  4k \sum_{l=0}^{k-1}  2^{k-l-1} \binom{k+l}{k-l-1} - 3\sum_{l=0}^{k-1}  (k-l-1) 2^{k-l-1} \binom{k+l}{k-l-1} \\
 &=& 4k \sum_{q=0}^{k-1}  2^{q} \binom{2k-1-q}{q} - 3\sum_{q=0}^{k-1}  q \, 2^{q} \binom{2k-1-q}{q} \\
 &=& \frac{4k (4^k-1)}{3} - \frac{2}{3} \left((4^k-1) (k- \frac{2}{3})-k \right) \\
 &=& \frac{1}{9} \left( (3c+1) 2^{c-1} - 4 \right) , \qquad c= 2k+1.
\end{eqnarray*}

The third term depends on the parity of $k$. 
In both cases, Lemma \ref{lem22} and Lemma \ref{lem23} play an important role to simplify the expression. 
If $k=2n$, then 
\begin{eqnarray*}
&& \sum_{\substack{l=0 \\ l + k +1 \equiv 0 \pmod{2}}}^{k-1} \frac{k+3l+3}{4}\, 2^{(k-l-1)/2}  \binom{(k+l-1)/2}{l} \\
&=& \sum_{p=0}^{n-1} \frac{n+3p+3}{2}\, 2^{n-p-1}  \binom{n+p}{2p+1} \\
&=& 2n  \sum_{p=0}^{n-1}  2^{n-p-1} \binom{n+p}{n-p-1}  -  \frac{3}{2} \sum_{p=0}^{n-1} (n-p-1)\, 2^{n-p-1} \binom{n+p}{n-p-1}  \\
&=& 2n   \sum_{q=0}^{n-1} 2^q \binom{2n-1-q}{q} -  \frac{3}{2} \sum_{q=0}^{n-1} q \, 2^{q} \binom{2n-1-q}{q} \\
&=& \frac{2n (4^n-1)}{3} -  3 \cdot \frac{n(4^n-2)- (4^n - 1) \frac{2}{3}}{9} \\
&=& \frac{2n(4^n-1)}{3} -  \frac{n(4^n-2)- (4^n - 1) \frac{2}{3}}{3} \\
&=& \frac{4^n-1}{3}  \left( n + \frac{2}{3} \right) + \frac{n}{3} \\
&=& \frac{1}{36} \left( (3c+5) \cdot 2^{(c-1)/2} - 8 \right), \qquad c = 4n+1 . 
\end{eqnarray*} 
Therefore if $c \equiv 1 \pmod{4}$, we conclude 
\begin{eqnarray*}
TG(c) 
&=& \frac{1}{8} \cdot \frac{1}{9} \left( (3c+1) 2^{c-1} - 4 \right) + \frac{1}{2} + \frac{1}{36} \left( (3c+5) \cdot 2^{(c-1)/2} - 8 \right) \\
&=& \frac{1}{36} \left( (3 c + 1) \cdot 2^{c-2} + (3 c + 5)\cdot 2^{(c-1)/2} + 8 \right) .
\end{eqnarray*}

If $k=2n+1$, then 
\begin{eqnarray*}
&& \sum_{\substack{l=0 \\ l + k +1 \equiv 0 \pmod{2}}}^{k-1} \frac{k+3l+3}{4}\, 2^{(k-l-1)/2}  \binom{(k+l-1)/2}{l}\\
&=& \sum_{p=0}^{n} \frac{n+3p+2}{2}\, 2^{n-p} \binom{n+p}{2p} \\
&=& (2n+1)  \sum_{p=0}^{n}  2^{n-p} \binom{n+p}{n-p}   -  \frac{3}{2} \sum_{p=0}^{n}  (n-p)\, 2^{n-p} \binom{n+p}{n-p} \\
&=& (2n+1)   \sum_{q=0}^{n}  2^{q} \binom{2n-q}{q}  -  \frac{3}{2} \sum_{q=0}^{n}   q \, 2^{q} \binom{2n-q}{q}\\
&=& (2n+1) \, \frac{2 \times 4^n+1}{3} -  \frac{3}{2} \, 2 \, \left[ \frac{4^n-1}{3}  (\frac{2n}{3} -\frac{1}{9})+ \frac{4n}{9} \right]\\
&=& \frac{4^n-1}{3}  \left( 2n + \frac{7}{3} \right) + \frac{2n}{3} +1 \\
&=& \frac{1}{36} \left( (3c+5) \cdot 2^{(c-1)/2} + 8 \right), \qquad c = 4n+3 . 
\end{eqnarray*} 
Therefore if $c \equiv 3 \pmod{4}$, we conclude 
\begin{eqnarray*}
TG(c) 
&=& \frac{1}{8} \cdot \frac{1}{9} \left( (3c+1) 2^{c-1} - 4 \right) + \frac{1}{2} + \frac{1}{36} \left( (3c+5) \cdot 2^{(c-1)/2} + 8 \right) \\
&=& \frac{1}{36} \left( (3 c + 1) \cdot 2^{c-2} + (3 c + 5)\cdot 2^{(c-1)/2} + 24 \right) .
\end{eqnarray*}

Hence we obtain Proposition \ref{totalgenus} for odd crossing number and 
this completes the proof of Proposition \ref{totalgenus}. 

\section{Average genus}

In this section, we show Theorem \ref{mainthm} by using Theorem \ref{thmernstsumners} and Proposition \ref{totalgenus}. 

\begin{proof}[Proof for Theorem \ref{mainthm}]
The average genus $\overline{g}_c$ is given by $TG(c) / TK(c)$. 
Both $TG(c)$ and $TK(c)$ are obtained by Theorem \ref{thmernstsumners} and Proposition \ref{totalgenus}. 

Case 1: for $c \equiv 0 \pmod{2}$,  
\begin{eqnarray*}
\overline{g}_c 
&=& \frac{\frac{1}{36} \left( (3 c + 1) 2^{c-2} - 16 \right)}{\frac{1}{3} \left( 2^{c-2} - 1 \right)} \\
&=& \frac{\left( (3 c + 1) 2^{c-2} - (3c+1) - 16 + (3c+1)  \right)}{12 \left( 2^{c-2} - 1 \right)} \\
&=& \frac{3c+1}{12} + \frac{3c - 15}{12 \left( 2^{c-2} - 1 \right)} \\
&=& \frac{c}{4} + \frac{1}{12} +\frac{c - 5}{2^c - 4} .
\end{eqnarray*}

Case 2: for $c \equiv 1 \pmod{4}$,  
\begin{eqnarray*}
\overline{g}_c 
&=& \frac{\frac{1}{36} \left( (3 c + 1) \cdot 2^{c-2} + (3 c + 5)\cdot 2^{(c-1)/2} + 8 \right)}{\frac{1}{3} \left(2^{c-2} + 2^{(c-1)/2} \right)} \\
&=& \frac{(3 c + 1) \cdot 2^{c-2} + (3 c + 1)\cdot 2^{(c-1)/2} + 4 \cdot 2^{(c-1)/2} + 8}{12 \left(2^{c-2} + 2^{(c-1)/2} \right)} \\
&=& \frac{3c+1}{12} + \frac{2^{(c-1)/2} + 2 }{3 \left(2^{c-2} + 2^{(c-1)/2} \right)} \\
&=& \frac{c}{4} + \frac{1}{12} + \frac{2^{(c-1)/2} + 2 }{3 \cdot 2^{(c-3)/2} \left(2^{(c-1)/2} + 2 \right)} \\
&=& \frac{c}{4} + \frac{1}{12} + \frac{1}{3 \cdot 2^{(c-3)/2} } .
\end{eqnarray*}

Case 3: for $c \equiv 3 \pmod{4}$,  
\begin{eqnarray*}
\overline{g}_c 
&=& \frac{\frac{1}{36} \left( (3 c + 1) \cdot 2^{c-2} + (3 c + 5)\cdot 2^{(c-1)/2} + 24 \right)}{\frac{1}{3} \left(2^{c-2} + 2^{(c-1)/2 + 2} \right)} \\
&=& \frac{\left( (3 c + 1) \cdot 2^{c-2} + (3 c + 1)\cdot 2^{(c-1)/2} + (3 c + 1)\cdot 2 + 4 \cdot 2^{(c-1)/2} - 6 c + 22 \right)}{12 \left(2^{c-2} + 2^{(c-1)/2} + 2 \right)} \\
&=& \frac{3c+1}{12} + \frac{2 \cdot 2^{(c-1)/2} - 3 c + 11}{6 \left(2^{c-2} + 2^{(c-1)/2} + 2 \right)} \\
&=& \frac{c}{4} + \frac{1}{12} + \frac{2^{(c+1)/2} - 3 c + 11}{12 \left(2^{c-3} + 2^{(c-3)/2} + 1 \right)} .
\end{eqnarray*}

\end{proof}

\section{Appendix}

In this section, a knot is regarded as the same as its mirror image. 
This implies that the following four $2$-bridge knots are considered as the same knot: 
\begin{align*}
&K([2a_1,2a_2, \ldots, 2a_{2m}]), K([-2a_1,-2a_2, \ldots, -2a_{2m}]), \\
&K([2a_{2m},\ldots,2 a_2, 2a_1]), K([-2a_{2m}, \ldots,- a_2, -2a_1]).
\end{align*}
If we do not distinguish a knot with its mirror image, 
then the number of $2$-bridge knots and the total genus, and the average genus are different from the above results. 
We determine the average genus of $2$-bridge knots with respect to crossing number, up to mirror image. 

Ernst and Sumners in \cite{ernstsumners} also gave the number of $2$-bridge knots with respect to crossing number, up to mirror image. 

\begin{theorem}[Ernst-Sumners \cite{ernstsumners}]\label{thmernstsumnersuptomirror}
For $c \geq 3$, the number of $2$-bridge knots with $c$ crossings up to mirror image is given by 
\[
TK^*(c) = 
\left\{
\begin{array}{ll}
\displaystyle{\frac{1}{3} (2^{c-3} +  2^{(c-4)/2})} & c \equiv 0 \pmod{4} 
\smallskip\\
\displaystyle{\frac{1}{3} (2^{c-3} +  2^{(c-3)/2})} & c \equiv 1 \pmod{4} 
\smallskip\\
\displaystyle{\frac{1}{3} (2^{c-3} + 2^{(c-4)/2} - 1)} & c \equiv 2 \pmod{4} 
\smallskip\\ 
\displaystyle{\frac{1}{3} (2^{c-3} + 2^{(c-3)/2} + 1)} & c \equiv 3 \pmod{4} 
\end{array}
\right. .
\]
\end{theorem}

We get the total genus of $2$-bridge knots with respect to crossing number up to mirror image, which is denoted by $TG^*(c)$ 
in a similar way. 

\begin{proposition}\label{totalgenusuptomirror}
For $c \geq 3$, we have 
\[
TG^*(c) =
\left\{
\begin{array}{ll}
\displaystyle{\frac{1}{72} \left( (3 c + 1) \cdot 2^{c-2} + (3 c + 2)\cdot 2^{(c-2)/2} - 8 \right)} & c \equiv 0 \pmod{4} 
\smallskip\\
\displaystyle{\frac{1}{72} \left( (3 c + 1) \cdot 2^{c-2} + (3 c + 5)\cdot 2^{(c-1)/2} + 8 \right)} & c \equiv 1 \pmod{4}
\smallskip \\
\displaystyle{\frac{1}{72} \left( (3 c + 1) \cdot 2^{c-2} + (3 c + 2)\cdot 2^{(c-2)/2} - 24 \right)} & c \equiv 2 \pmod{4}
\smallskip \\ 
\displaystyle{\frac{1}{72} \left( (3 c + 1) \cdot 2^{c-2} + (3 c + 5)\cdot 2^{(c-1)/2} + 24 \right)} & c \equiv 3 \pmod{4} 
\end{array}
\right. .
\]
\end{proposition} 

\begin{proof}[Sketch of Proof] 
First, we consider the case that the crossing number is even. 
Similarly, we can count the number of all the $2$-bridge knots with respect to $\ell = 2 l, c = 2k$, and $m$. 
Then the total number of the genera of all the $2$-bridge knots with crossing number $c=2k$ is given by 
\begin{eqnarray*}
TG^*(c) 
&=& \sum_{l=0}^{k-1} \sum_{m=l+1}^{(k+l)/2}  \frac{m}{2} \binom{k+l-1}{2m-1} \binom{2m-1}{2l} \\
&& + \,  \sum_{\substack{l=0 \\ l + k \equiv 0 \pmod{2}}}^{k-1} \sum_{m=l+1}^{(k+l)/2} \frac{m}{2} \binom{(k+l-2)/2}{m-1} \binom{m-1}{l} \\
&=& \sum_{l=0}^{k-1} \sum_{m=l+1}^{(k+l)/2}  \frac{m}{2} \binom{k+l-1}{2l} \binom{k-l-1}{2m-2l-1} \\
&& + \,  \sum_{\substack{l=0 \\ l + k \equiv 0 \pmod{2}}}^{k-1} \sum_{m=l+1}^{(k+l)/2} \frac{m}{2} \binom{(k+l-2)/2}{l} \binom{(k-l-2)/2}{m-l-1}.
\end{eqnarray*}
Moreover we can simplify the above expression in a similar way and get the desired statement. 

Besides, it is easy to see that $TG^*(2k+1) = \frac{1}{2} TG(2k+1)$, 
since knots with odd crossing numbers are not amphicheiral knots. 
\end{proof}

Therefore we obtain the average of genus of $2$-bridge knots with respect to  crossing number up to mirror image, which is denoted by $\overline{g}^*(c)$. 

\begin{theorem}\label{mainthmuptomirror}
For $c \geq 3$, we have
\[
\overline{g}^*_{c} =
\left\{
\begin{array}{ll}
\displaystyle{\frac{1}{4} c + \frac{1}{12} +  \frac{2^{(c-4)/2} - 4}{3 \left(2^{c-1} + 2^{c/2} \right)}} & c \equiv 0 \pmod{4} 
\smallskip\\
\displaystyle{\frac{1}{4} c + \frac{1}{12} +  \frac{1}{3 \cdot 2^{(c-3)/2}}} & c \equiv 1 \pmod{4} 
\smallskip\\
\displaystyle{\frac{1}{4} c + \frac{1}{12} +  \frac{2^{(c-4)/2} + 3 c - 11}{12 \left(2^{c-3} + 2^{(c-4)/2} - 1\right)}} & c \equiv 2 \pmod{4} \smallskip\\
\displaystyle{\frac{1}{4} c + \frac{1}{12} +  \frac{2^{(c+1)/2} - 3 c + 11}{12 \left(2^{c-3} + 2^{(c-3)/2} + 1\right)}} & c \equiv 3 \pmod{4} 
\end{array}
\right. .
\]
In particular, $\overline{g}^*_c \sim \displaystyle{ \frac{1}{4} c + \frac{1}{12}}$ as $c \to \infty$. 
\end{theorem}

It is easy to see that $\overline{g}_c = \overline{g}^*_c$ for odd crossing numbers, 
since $TK^*(2k+1) = \frac{1}{2} TK(2k+1)$ and $TG^*(2k+1) = \frac{1}{2} TG(2k+1)$. 
As a final example, we show Table \ref{tablegc} on the values which are provided in this paper.  

\renewcommand{\arraystretch}{1.3}
\begin{table}[h]
\begin{tabular}{c||c|c|c|c|c|c|c|c|c|c|c|c|c}
$c$ & 3 & 4 & 5 & 6 & 7 & 8 & 9 & 10 & 11 & 12 & 13 & 14 & 15  \\
\hline
\hline
$TK(c)$ & 2 & 1 & 4 & 5 & 14 & 21 & 48 & 85 & 182 & 341 & 704 & 1365 & 2774  \\
\hline
$TG(c)$ & 2 & 1 & 6 & 8 & 26 & 44 & 114 & 220 & 518 & 1052 & 2354 & 4892 & 10646  \\
\hline
$\overline{g}_{c}$ & $1$ & $1$ & $\frac{3}{2}$ & $\frac{8}{5}$ & $\frac{13}{7}$ & $\frac{44}{21}$ 
& $\frac{19}{8}$ & $\frac{44}{17}$ & $\frac{37}{13}$ & $\frac{1052}{341}$ & $\frac{107}{32}$ & $\frac{4892}{1365}$ & $\frac{5323}{1387}$  \\
\hline
\hline
$TK^*(c)$ & 1 & 1 & 2 & 3 & 7 & 12 & 24 & 45 & 91 & 176 & 352 & 693 & 1387 \\
\hline
$TG^*(c)$ & 1 & 1 & 3 & 5 & 13 & 25 & 57 & 117 & 259 & 543 & 1177 & 2485 & 5323 \\
\hline
$\overline{g}^*_{c}$ & 1 & 1 & $\frac{3}{2}$ & $\frac{5}{3}$ & $\frac{13}{7}$ & $\frac{25}{12}$ 
& $\frac{19}{8}$ & $\frac{13}{5}$ & $\frac{37}{13}$ & $\frac{543}{176}$ & $\frac{107}{32}$ & $\frac{355}{99}$ & $\frac{5323}{1387}$ \\
\end{tabular}
\bigskip
\caption{$\overline{g}_{c}$ and $\overline{g}^*_{c}$ }
\label{tablegc}
\end{table}
\renewcommand{\arraystretch}{1}

\section*{Acknowledgments}
The first author is partially supported by KAKENHI grant No.\ 20K03596, 19H01785, and 21H00986 
from the Japan Society for the Promotion of Science.  The second author is supported by grants from the Simons Foundation (\#354595 and \#708778).

\end{document}